\documentclass[11pt]{article}
\usepackage{epsf}
\usepackage{amsbsy,amsmath}
\usepackage{mathtools}
\usepackage{mathrsfs}
\usepackage{amsfonts}
\usepackage{breqn}
\usepackage{amssymb}
\usepackage{enumitem}
\usepackage{eucal}
\usepackage{graphics,mathrsfs}
\usepackage{amsthm}
\usepackage{secdot}
\newtheorem{theorem}{Theorem}[section]
\newtheorem{lemma}[theorem]{Lemma}

\theoremstyle{definition}
\newtheorem{definition}[theorem]{Definition}

\theoremstyle{remark}

\numberwithin{equation}{section}

\numberwithin{equation}{section}

\newsavebox{\savepar}

\begin{document}

\title{Three nontrivial solutions of a nonlocal problem involving critical exponent}
\author{Amita Soni and D.Choudhuri}

\date{}
\maketitle

\begin{abstract}
\noindent In this paper we will prove the existence of three nontrivial weak solutions of the following problem involving a nonlinear integro-differential operator and a term with critical exponent. 
\begin{align*}
\begin{split}
-\mathscr{L}_\Phi u & = |u|^{{p_{s}^{\ast}}-2}u+\lambda f(x,u)\,\,\mbox{in}\,\,\Omega,\\
u & = 0\,\, \mbox{in}\,\, \mathbb{R}^N\setminus \Omega,
\end{split}
\end{align*}
Here $q\in(p, p_s^*)$, where $p_s^*$ is the fractional Sobolev
conjugate of $p$ and $-\mathscr{L}_\Phi $ represents a general nonlocal integro-differential operator of order $s\in(0,1)$. This operator is possibly degenerate and covers the case of fractional $p$-Laplacian operator.
\end{abstract}
\begin{flushleft}

\end{flushleft}

\section{Introduction}
In recent years, problems involving critical exponent is gaining a lot of attention because of the difficulty in dealing with the loss of compactness. A fundamental and important tool to deal with such type of loss is concentration-compactness principle by P.L.Lions \cite{Lions}. An application of this principle can be seen in the paper of Azorero et. al \cite{Azorero} who to our knowledge were the first to work in this direction. Here in \cite{Azorero}, the authors have shown the existence of solutions for different cases of $\lambda$ involved in the equation along with multiplicity of solutions of noncritical problem with non symmetric functional. An advancement of this principle can be seen in \cite{Bonder} where authors have obtained concentration-compactness lemma for fractional $p$-Laplacian. An immediate subject of interest to the researchers is about multiplicity of solutions of such kind of problems. One of the methods to get the multiple solutions is based on the concepts involving manifolds. It was Struwe \cite{Struwe} who worked on such problems defined on manifolds to guarantee the existence of three nontrivial weak solutions for a Pseudo-Laplace-operator without critical exponent term. Following the work of Struwe \cite{Struwe}, there are a few articles devoted to the existence of multiplicity of solutions by combining the concept of concentration-compactness lemma and manifolds. An important such contribution can be found in \cite{Napoli}. Here \cite{Napoli} the authors have proved the existence of three nontrivial solutions of problem involving $p$-Laplacian and critical exponent on manifolds. In \cite{Bonder1} and \cite{Silva} the authors have extended this result for $p(x)$ Laplacian operator. The authors in \cite{Cantizano} obtained the same results for the problem involving a fractional $p$-Laplacian with a critical exponent term.

    Motivated by the above papers, in this paper we are going to extend our result in a more general setting, i.e. for a general nonlocal integro-differential operator. It is a generalization owing to the fact that it covers the case of fractional $p$-Laplacian. This type of operator appeared in many applications and models such as in quasi-geostrophic dynamics, nonlocal diffusion and modified porous medium equations, dislocation problems, phase transition models (refer \cite{Caffarelli}, \cite{caf1} and references therein). For more information on this operator the readers may refer to \cite{kussi}. Note that we will not impose the condition that $f$ is odd. The problem we will address is as follows.
\begin{align}
\begin{split}
-\mathscr{L}_\Phi u & = |u|^{{p_{s}^{\ast}}-2}u+\lambda f(x,u)\,\,\mbox{in}\,\,\Omega,\\
u & = 0\,\, \mbox{in}\,\, \mathbb{R}^N\setminus \Omega,
\end{split}
\end{align}
\noindent where, we have the following assumptions on the term $f(x,u)$.

\noindent $(f_{1})$ $f:\Omega\times\mathbb{R}\rightarrow \mathbb{R}$ is a measurable function with respect to the first variable and continuously differentiable with respect to second variable for almost every $x \in \Omega$. Moreover $f(x,0)=0$ for every $x \in \Omega$.

\noindent $(f_{2})$ There exist constants $c_{1} \in \left(0,\frac{1}{p_{s}^{\ast}-1}\right), c_{2}\in (p,p_{s}^{\ast}), 0 < c_{3} < c_{4}$ such that for any $u \in L^{q}(\Omega)$ and $p < q < p_{s}^{\ast},$

\noindent $c_{3}||u||_{L^{q}(\Omega)}^{q}\leq c_{2}\int_{\Omega}F(x,u)dx\leq\int_{\Omega}f(x,u)u dx\leq c_{1}\int_{\Omega}f_{u}(x,u)u^{2}dx\leq c_{4}||u||_{L^{q}(\Omega)}^{q}$.\\
\noindent The main result of this article is given in the form of the following theorem.

\begin{theorem}\label{Mainthm}
With the assumptions $(f_{1})$ and $(f_{2})$ there exist $\lambda^{\prime} > 0$ depending on $n, p, p^{\ast}_{s}$ and the constant $c_{3}$ in $(f_{2})$ such that for every $\Lambda \in \left[1,(\frac{c_{2}}{p})^{1/4}\right)$, where $c_{2}$ is also from $(f_{2})$ and $\lambda > \lambda^{\prime}$, there exist three distinct nontrivial weak solutions of the problem $(P)$. Moreover one of these solutions is positive, one is negative and one has a non-constant sign.
\end{theorem}

\section{Functional analytic setup and Main tools}
We now discuss a few definitions and  the functional analytic setting which will be used in this paper. $-\mathscr{L}_\Phi$ is a nonlocal operator 
which is defined as
\begin{equation}
\langle-\mathscr{L}_\Phi u, \varphi\rangle
=\int_{\mathbb{R}^N}\int_{\mathbb{R}^N}
\Phi(u(x)-u(y))(\varphi(x)-\varphi(y))K(x,y)dxdy,\label{e3}
\end{equation}
for every smooth function $\varphi$ with compact support, i.e.,
$\varphi \in C_c^\infty(\mathbb{R}^N)$. The function $\Phi$ is a
real valued continuous function over $\mathbb{R}$, satisfying
$\Phi(0)=0$ together with the following monotonicity property
\begin{equation}
\Lambda ^{-1} |t|^p \leq \Phi(t)t\leq \Lambda |t|^p, \,\,\forall
\,t\in \mathbb{R}.\label{e4}
\end{equation}
We assume that $\Phi$ is convex to get the following condition.
\begin{equation}
a^{-1}|t|^{p}\leq \Phi^{\prime}(t)t^{2}\leq a|t|^{p}
\end{equation}
for some $a > 0$. The kernel $K : \mathbb{R}^N \times \mathbb{R}^N \rightarrow
\mathbb{R}$ is a measurable function satisfying the following
ellipticity property
\begin{equation}
\frac{1}{\Lambda |x-y|^{N+sp}} \leq K(x,y) \leq
\frac{\Lambda}{|x-y|^{N+sp}}, \,\forall x, y \in \mathbb{R}^N, x\neq
y,\label{e5}
\end{equation}
where $\Lambda \geq 1$, $s\in (0,1),\; p > 2- \frac{s}{N}$ with $
q\in(p,p_s^{*})$, where $p_s^{*}= \frac{Np}{N-sp}$ and $p'= \frac{p}{p-1}$. Assumptions made in $(\ref{e4})$ and $(\ref{e5})$
makes the nonlocal operator $-\mathscr{L}_\Phi$ to be an elliptic
operator. Note that, upon taking the special case $\Phi(t)=
|t|^{p-2}t$ with $K(x,y)= |x-y|^{-(N+sp)}$ in $(\ref{e3})$, we
recover the fractional
$p$-Laplacian, for every $\varphi \in C_{c}^{\infty}(\mathbb{R}^N)$ as 
$$\langle (-\Delta)_p^su, \varphi\rangle=\displaystyle{\int_{\mathbb{R}^N}\int_{\mathbb{R}^N}\frac{|u(x)-u(y)|^{p-2}(u(x)-u(y))(\varphi(x)-\varphi(y))}{|x-y|^{N+ps}}dxdy}$$ 
In other words, the nonlocal operator $-\mathscr{L}_\Phi$, is a
generalization of the fractional $p$-Laplacian for $1 \leq p < \infty$ and $s\in (0,1)$. Due to the nonlocal character of
$\mathscr{L}_\Phi$ defined in $\ref{e3}$, it is natural to work with
Sobolev space $W^{s, p}(\mathbb{R}^N)$ and express the Dirichlet
condition in $\mathbb{R}^N\setminus \Omega$ rather than
$\partial\Omega$. Although fractional Sobolev space are well known
since the beginning of the last century, especially in the field of
harmonic analysis, they have become increasingly popular in the last
few years due to the immense work of Caffarelli \& Silvestre
\cite{Caffarelli} and the references therein. We now turn to our problem for which we provide the variational setting on a suitable function space for $(1.1)$, in conjunction with some
preliminary results. For all measurable functions $u:
\mathbb{R}^N\rightarrow \mathbb{R}$, we set
$$\|u\|_{L^p(\mathbb{R}^N)}=\left(\int_{\mathbb{R}^N}|u(x)|^p dx\right)^{\frac{1}{p}},$$
$$[u]_{s,p}= \left(\int_{\mathbb{R}^{2N}}\frac{|u(x)-u(y)|^p}{|x-y|^{N+sp}}
dxdy\right)^{\frac{1}{p}},$$ where $p\in(1, \infty)$ and $s\in
(0,1)$. The fractional Sobolev space $W^{s,p}(\mathbb{R}^N)$ is
defined as the space of all function $u\in L^p(\mathbb{R}^N)$ such
that $[u]_{s,p}$ is finite and endowed with the norm
$$\|u\|_{W^{s,p}(\mathbb{R}^N)}=
\left(\|u\|_{L^p(\mathbb{R}^N)}^p+[u]_{s,p}^p\right)^{\frac{1}{p}}.$$
Further information on fractional Sobolev space can be found in Nezza et al. \cite{nezza} and the references therein. We now define $W_{0}^{s,p}(\Omega)$ which is a closed linear
subspace of $W^{s,p}(\mathbb{R}^N)$ as follows. \\
$$W_0^{s,p}(\Omega)= \{u\in W^{s,p}(\mathbb{R}^N): u=0\,\,
a.e.\,\,\text{in}\,\,\mathbb{R}^N\setminus\Omega\}.$$ It is easy
to see that the norms $\|\cdot\|_{L^p(\mathbb{R}^N)}$ and $\|\cdot
\|_{L^p(\Omega)}$ agree on $W_0^{s,p}(\Omega)$. We also have a
Poincar\'{e} type inequality which is as follows:
$$\|u\|_{L^p(\Omega)}\leq C
[u]_{s,p},\,\,\,\mbox{for\,\,all}\,\,u\in W_0^{s,p}(\Omega).$$ Thus,
we can equivalently renorm $W_0^{s,p}(\Omega)$, by setting
$\|u\|_{W_0^{s,p}(\Omega)}= [u]_{s,p}$, for every $u\in
W_0^{s,p}(\Omega)$. Let $p_s^*= \frac{Np}{N-sp}$ , with the
agreement that $p_s^*=\infty$ if $N\geq sp$. It is well known that
$(W_0^{s,p}(\Omega), \|\cdot\|_{W_0^{s,p}(\Omega)})$ is a uniformly
convex reflexive Banach space. It is continuously embedded into
$L^r(\Omega)$, for all $r\in[1, p_s^*]$ if $sp<N$, for all $1\leq r
<\infty$ if $N=sp$ and into $L^\infty(\Omega)$ if $N<sp$. It is also
compactly embedded in $L^r(\Omega)$ for any $r\in[1, p_s^*)$ if
$N\geq sp$ and in $L^\infty(\Omega)$ for $N<sp$. Furthermore,
$C_c^\infty(\Omega)$ is a dense subspace of $W_0^{s,p}(\Omega)$ with
respect to $\|\cdot\|_{W^{s,p}(\Omega)}$.\\
We define an associated energy functional to the problem $(1.1)$ as
$$I(u) =
\int_{\mathbb{R}^N}\int_{\mathbb{R}^N}
\mathcal{P}\Phi(u(x)-u(y))K(x,y)dx dy -\frac{1}{p_{s}^{\ast}}\int_\Omega |u|^{p_{s}^{\ast}}-\lambda \int_{\Omega} F(x,u(x))dx,$$ where $\mathcal{P}\Phi(t) :=
\int_0^{|t|}\Phi(\tau)d\tau$ being the primitive of $\Phi$. Thus by
$(\ref{e4})$ we have
\begin{equation} \Lambda^{-1}\frac{|t|^p}{p}\leq
\mathcal{P}\Phi(t)\leq \Lambda \frac{|t|^p}{p},\label{e6}
\end{equation}
for $t\neq 0$ and $\mathcal{P}\Phi(0) =0$. The Fr\'{e}chet
derivative of $I_{P_1}$, which is in $W_0^{-s, p'}(\Omega)$, the dual
space of $W_0^{s, p}(\Omega)$ where $p'= \frac{p}{p-1}$ is
defined as
\begin{align}
\begin{split}
\langle I'(u), v\rangle = \int_{\mathbb{R}^N}\int_{\mathbb{R}^N}
\Phi(u(x)-u(y))(v(x)-v(y))K(x,y)dxdy -\int_\Omega |u|^{{p^{\ast}_{s}}-2}uvdx\\
\MoveEqLeft\hspace{-11.5cm}-\lambda \int_{\Omega}f(x,u)v dx \;\;\forall\; v\in W_0^{s, p}(\Omega).\label{e8'}
\end{split}
\end{align}
We will now give a few definitions which will be used in this article.
\begin{definition}\label{defn1}
For i=1,2,3, let us define the manifold $M_{i} \subset W_{0}^{s,p}(\Omega)$ as
\begin{align*}
\begin{split}
M_{1}=&\lbrace u \in W_{0}^{s,p}(\Omega):\int_{\Omega}u_{+} > 0 \;\mathrm{and}\;\int_{\mathbb{R}^N}\int_{\mathbb{R}^N}
\Phi(u_{+}(x)-u_{+}(y))(u_{+}(x)-u_{+}(y))\\
&K(x,y)dxdy-\int_{\Omega} |u_{+}|^{p^{\ast}_{s}}dx=\lambda \int_{\Omega}f(x,u)u_{+} dx\rbrace,\\
M_{2}=&\lbrace u \in W_{0}^{s,p}(\Omega):\int_{\Omega}u_{-} > 0 \;\mathrm{and}\;\int_{\mathbb{R}^N}\int_{\mathbb{R}^N}
\Phi(u_{-}(x)-u_{-}(y))(u_{-}(x)-u_{-}(y))\\
&K(x,y)dxdy-\int_{\Omega} |u_{-}|^{p^{\ast}_{s}}dx=\lambda \int_{\Omega}f(x,u)u_{-} dx\rbrace,\\
M_{3}=&M_{1}\cap M_{2},
\end{split}
\end{align*}
where $u_{+}=\text{max}\lbrace u,0\rbrace$ and $u_{-}=\text{max}\lbrace -u,0\rbrace$.
\end{definition}

\begin{definition}\label{defn2}
For i=1,2,3, let us define $K_{i} \subset W_{0}^{s,p}(\Omega)$ as
$K_{1}=\lbrace u\in M_{1}: u \geq 0\rbrace, K_{2}=\lbrace u \in M_{2}: u \leq 0\rbrace, K_{3}=M_{3}$.
\end{definition}

\begin{definition}\label{defn3}
(Ekeland Variational Principle) Let $\Phi$ be a lower semicontinuous bounded below function from a Banach space X into $\mathbb{R} \cup \left\lbrace+\infty\right\rbrace$. For every $\epsilon > 0$, there is $x_{0} \in X$ such that $\Phi(x) \geq \Phi(x_{0})-\epsilon{\|x-x_{0}\|}$ for every $x \in X$.
\end{definition}

 \section{Existence Results}
In order to prove the main result of this article, given in the form of Theorem (1.1), we will first prove a few lemmas.  

\begin{lemma}\label{lem1}
For every $w_{0} \in W_{0}^{s,p}(\Omega), w_{0} > 0 (w_{0} < 0),$ there exists $t_{\lambda} > 0$ such that $t_{\lambda}w_{0} \in M_{1} (\in M_{2})$. Moreover $\underset{\lambda\rightarrow\infty}{lim}t_{\lambda}=0$.\\
       consequently, given $w_{0},w_{1} \in W_{0}^{s,p}(\Omega), w_{0} > 0, w_{1} < 0$ with disjoint supports, there exist $\overline{t_{\lambda}},\underline{t_{\lambda}}$ such that $\overline{t_{\lambda}}w_{0}+\underline{t_{\lambda}}w_{1} \in M_{3}$ and 
$\overline{t_{\lambda}},\underline{t_{\lambda}}\rightarrow 0$ as $\lambda\rightarrow \infty$.
\end{lemma}
\begin{proof}
It is enough to prove the lemma for $M_{1}$ since the other cases will follow analogously.\\
For $w\in W_{0}^{s,p}(\Omega)$ with $w \geq 0$, consider the functional \\
$$\varphi_{1}(w)=\int_{\mathbb{R}^N}\int_{\mathbb{R}^N}
\Phi(w(x)-w(y))(w(x)-w(y))K(x,y)dxdy -\int_\Omega |w|^{p^{\ast}_{s}}dx-\lambda \int_{\Omega}f(x,w)w dx.$$\\
In order to prove the lemma we will show that for each $w_{0}\in W_{0}^{s,p}(\Omega)$ there exists a `$t_{\lambda}$' such that $\varphi_{1}(t_{\lambda}w_{0})=0$.
\begin{align}
\begin{split}
\varphi_{1}(tw_{0})&=\int_{\mathbb{R}^N}\int_{\mathbb{R}^N}
\Phi(tw_{0}(x)-tw_{0}(y))(tw_{0}(x)-tw_{0}(y))K(x,y)dxdy
 -\int_\Omega |tw_{0}|^{p^{\ast}_{s}}dx\\
 &-\lambda \int_{\Omega}f(x,tw_{0})tw_{0} dx
\end{split}
\end{align}
\begin{align*}
\begin{split}
&\geq \Lambda^{-2}\int_{\mathbb{R}^N}\int_{\mathbb{R}^N}\frac{|tw_{0}(x)-tw_{0}(y)|^{p}}{|x-y|^{N+ps}}dx dy-t^{p_{s}^{\ast}}\int_\Omega |w_{0}|^{p^{\ast}_{s}}dx-\lambda \int_{\Omega}f(x,tw_{0})tw_{0} dx\\
&=\Lambda^{-2}t^{p}[w_{0}]_{s,p}^{p}-t^{p_{s}^{\ast}}\int_\Omega |w_{0}|^{p^{\ast}_{s}}dx-\lambda \int_{\Omega}f(x,tw_{0})tw_{0} dx
\end{split}
\end{align*}
Further on using $(f_{2})$, we get
\begin{align}
\begin{split}
\varphi_{1}(tw_{0})&\geq \Lambda^{-2}t^{p}[w_{0}]_{s,p}^{p}-t^{p_{s}^{\ast}}\int_\Omega |w_{0}|^{p^{\ast}_{s}}dx-\lambda c_{4} \int_{\Omega}|tw_{0}|^{q} dx\\
&\geq A_{1}t^{p}-Bt^{p_{s}^{\ast}}-E_{1}t^{q}
\end{split}
\end{align}
where, $A_{1}=\Lambda^{-2}[w_{0}]_{s,p}^{p}, B=\int_{\Omega}|w_{0}|^{p_{s}^{\ast}}dx, E_{1}=\lambda c_{4}\int_{\Omega}|w_{0}|^{q}dx$
and,
\begin{align}
\begin{split}
\varphi_{1}(tw_{0})&\leq \Lambda^{2}t^{p}[w_{0}]_{s,p}^{p}-t^{p_{s}^{\ast}}\int_\Omega |w_{0}|^{p^{\ast}_{s}}dx-\lambda c_{3} \int_{\Omega}|tw_{0}|^{q} dx\\
&\leq A_{2}t^{p}-Bt^{p_{s}^{\ast}}-E_{2}t^{q}
\end{split}
\end{align}
where, $A_{2}=\Lambda^{2}[w_{0}]_{s,p}^{p}, B=\int_{\Omega}|w_{0}|^{p_{s}^{\ast}}dx, E_{2}=\lambda c_{3}\int_{\Omega}|w_{0}|^{q}dx$\\
Since, $p < q < p_{s}^{\ast}$, it follows from (3.1) and (3.2) that $\varphi_{1}(tw_{0}) > 0$ for a $`t'$ small enough and $\varphi_{1}(tw_{0}) < 0$ for some $`t'$ big enough. Hence there exists some $`t_{\lambda}'$ such that $\varphi_{1}(t_{\lambda}w_{0})=0$. We will now establish the existence of an upper bound of $t_{\lambda}$. For this it is enough to show the existence of some $t_{1}$ such that $\varphi_{1}(t_{1}w_{0}) < 0$. From (3.2), it is easy to see that
$\varphi_{1}(t_{1}w_{0}) <  A_{2}t^{p}-E_{2}t^{q}$. On choosing $t_{1}$ such that $A_{2}t_{1}^{p}-E_{2}t_{1}^{q}=0$ implying $t_{1}=\left(\frac{A_{2}}{E_{2}}\right)^{\frac{1}{q-p}}$. Therefore, there exists $t_{\lambda}\in [0,t_{1}]$. The definition of $\varphi_{1}$ in (3.1) implies that $t_{\lambda}\rightarrow 0$ as $\lambda \rightarrow \infty$.
\end{proof}

\begin{lemma}\label{lem2}
There exist constants $\alpha_{j} > 0$ such that for every $u\in K_{i},$\; i=1,2,3\; $\alpha_{1}[u]_{s,p}^{p}\leq \alpha_{2}\left(\int_{\Omega}|u|^{p_{s}^{\ast}}dx+\int_{\Omega}\lambda f(x,u)udx\right)\leq \alpha_{3}I(u)\leq \alpha_{4}[u]_{s,p}^{p}$.
\end{lemma}
\begin{proof}
As $u\in K_{i}$ so 
\begin{align}
\begin{split}
\int_{\mathbb{R}^N}\int_{\mathbb{R}^N}
\Phi(u(x)-u(y))(u(x)-u(y))K(x,y)dxdy =\int_\Omega |u|^{p^{\ast}_{s}}dx+\lambda \int_{\Omega}f(x,u)u dx
\end{split}
\end{align}
Using (2.2), we get $\Lambda^{-2}[u]_{s,p}^{p}\leq \int_\Omega |u|^{p^{\ast}_{s}}dx+\lambda \int_{\Omega}f(x,u)u dx$. On choosing $\alpha_{2}=\Lambda^{2}\alpha_{1}$, first inequality holds.\\
We now prove the last inequality i.e., $\alpha_{3}I(u)\leq \alpha_{4}[u]_{s,p}^{p}$. Using $(f_{2})$, (2.2) and (3.4) we get,
\begin{align*}
\begin{split}
\left|\lambda\int_{\Omega}F(x,u)dx\right|&=\lambda\int_{\Omega}F(x,u)dx\leq \lambda\frac{1}{c_{2}}\int_{\Omega}f(x,u)u dx\\
&=\frac{1}{c_{2}}\left(\Phi(u(x)-u(y))(u(x)-u(y))K(x,y)dxdy -\int_\Omega |u|^{p^{\ast}_{s}}dx\right)\\
&\leq\frac{\Lambda^{2}}{c_{2}}[u]^{p}_{s,p}-\frac{1}{c_{2}}\int_{\Omega}|u|^{p_{s}^{\ast}}
\end{split}
\end{align*}
From here we obtain that,
\begin{align} 
-\lambda\int_{\Omega}F(x,u)dx\leq \frac{1}{c_{2}}\left[\Lambda^{2}[u]^{p}_{s,p}-\int_{\Omega}|u|^{p^{\ast}_{s}}\right]
\end{align}
 So, using (3.5) and (2.5)
\begin{align*}
\begin{split}
I(u)&=\int_{\mathbb{R}^N}\int_{\mathbb{R}^N}
\mathcal{P}\Phi(u(x)-u(y))K(x,y)dx dy -\frac{1}{p_{s}^{\ast}}\int_\Omega |u|^{p_{s}^{\ast}}-\lambda \int_{\Omega} F(x,u)dx\\
&\leq \frac{\Lambda^{2}}{p}[u]_{s,p}^{p}-\frac{1}{p^{\ast}_{s}}\int\Omega |u|^{p_{s}^{\ast}}+\frac{1}{c_{2}}\left[\Lambda^{2}[u]_{s,p}^{p}-\int_\Omega |u|^{p_{s}^{\ast}}\right]\\
&\leq \left(\frac{\Lambda^{2}}{p}+\frac{\Lambda^{2}}{c_{2}}\right)[u]_{s,p}^{p}=\Lambda^{2}\left(\frac{1}{p}+\frac{1}{c_{2}}\right)[u]_{s,p}^{p}
\end{split}
\end{align*}
On choosing $\alpha_{4}=\Lambda^{2}\left(\frac{1}{p}+\frac{1}{c_{2}}\right)\alpha_{3}$, we obtain the inequality $\alpha_{3}I(u)\leq \alpha_{4}[u]_{s,p}^{p}$. What we are left to prove is the middle inequality i.e. $\alpha_{2}\left(\int_{\Omega}|u|^{p_{s}^{\ast}}dx+\int_{\Omega}\lambda f(x,u)udx\right)\leq \alpha_{3}I(u)$. For this, we again use $(f_{2})$ and (2.5) to get
\begin{align*}
\begin{split}
I(u)&=\int_{\mathbb{R}^N}\int_{\mathbb{R}^N}
\mathcal{P}\Phi(u(x)-u(y))K(x,y)dx dy -\frac{1}{p_{s}^{\ast}}\int_\Omega |u|^{p_{s}^{\ast}}-\lambda \int_{\Omega} F(x,u)dx\\
&\geq \frac{\Lambda^{-2}}{p}[u]_{s,p}^{p}-\frac{1}{p^{\ast}_{s}}\int_{\Omega} |u|^{p_{s}^{\ast}}-\frac{1}{c_{2}}\lambda\int_{\Omega}f(x,u)u dx
\end{split}
\end{align*}
Multiplying both sides by $c_{2}$ and using (2.2), (2.4),
\begin{align*}
\begin{split}
c_{2}I(u)&\geq \frac{c_{2}\Lambda^{-2}}{p}[u]_{s,p}^{p}-\frac{c_{2}}{p^{\ast}_{s}}\int_{\Omega} |u|^{p_{s}^{\ast}}-\lambda\int_{\Omega}f(x,u)u dx\\
&\geq \frac{c_{2}\Lambda^{-4}}{p}\int_{\mathbb{R}^N}\int_{\mathbb{R}^N}
\Phi(u(x)-u(y))(u(x)-u(y))K(x,y)dxdy-\frac{c_{2}}{p^{\ast}_{s}}\int_{\Omega} |u|^{p_{s}^{\ast}}dx\\
&-\lambda\int_{\Omega}f(x,u)u dx\\
&= \frac{c_{2}\Lambda^{-4}}{p}\left(\int_{\Omega}|u|^{p_{s}^{\ast}}dx+\lambda\int_{\Omega}f(x,u)u dx\right)-\frac{c_{2}}{p^{\ast}_{s}}\int_{\Omega} |u|^{p_{s}^{\ast}}-\lambda\int_{\Omega}f(x,u)u dx\\
&= c_{2}\left(\frac{\Lambda^{-4}}{p}-\frac{1}{p_{s}^{\ast}}\right) \int_{\Omega}|u|^{p_{s}^{\ast}}dx+\lambda\left(\frac{c_{2}\Lambda^{-4}}{p}-1\right)\int_{\Omega}f(x,u)u dx
\end{split}
\end{align*}
Since, from assumption $(f_{2})$ we know $c_{2} < p^{\ast}_{s}$, hence
\begin{align*}
c_{2}I(u) \geq\left(\frac{c_{2}\Lambda^{-4}}{p}-1\right)\left(\int_{\Omega} |u|^{p_{s}^{\ast}}dx+\lambda\int_{\Omega}f(x,u)u dx\right) 
\end{align*}
Now on choosing $\alpha_{3}=c_{2}$ and $\alpha_{2}=\left(\frac{c_{2}\Lambda^{-4}}{p}-1\right);$ where $\Lambda\in \left[1,\left( \frac{c_{2}}{p})^{\frac{1}{4}}\right) \right]$, the inequality follows.
\end{proof}

\begin{lemma}\label{lem3}
There exists a constant $D > 0$ such that $[u_{+}]_{s,p}^{p} \geq D$ in $K_{1}$, $[u_{-}]_{s,p}^{p} \geq D$ in $K_{2}$ and $[u_{+}]_{s,p}^{p}, [u_{-}]_{s,p}^{p} \geq D$ in $K_{3}$.
\end{lemma}

\begin{proof}
From the definition of $K_{i}$, we have
\begin{align}
\int_{\mathbb{R}^N}\int_{\mathbb{R}^N}
\Phi(u_{\pm}(x)-u_{\pm}(y))(u_{\pm}(x)-u_{\pm}(y))K(x,y)dxdy =\int_\Omega |u_{\pm}|^{p^{\ast}_{s}}dx+\lambda \int_{\Omega}f(x,u)u_{\pm} dx
\end{align}
Using $(f_{2})$, (2.2), (2.4) and Sobolev embedding in (3.6) we get,
\begin{align*}
\begin{split}
\Lambda^{-2}[u_{\pm}]_{s,p}^{p}&\leq {\|u_{\pm}\|}^{p^{\ast}_{s}}_{p^{\ast}_{s}}+c_{4}\|u_{\pm}\|^{q}_{q}\leq \tilde{C}[u_{\pm}]_{s,p}^{r}
\end{split}
\end{align*}
where $r=q$ if $[u_{\pm}]_{s,p} < 1$ or $r=p_{s}^{\ast}$ if $[u_{\pm}]_{s,p}\geq 1$. Hence using the fact that $p < r$, we get the result.
\end{proof}

\begin{lemma}\label{lem4}
For manifolds $M_{i}$ and $K_{i}$ the following holds:\\
(i) $M_{i}$ is a $C^{1}$ sub-manifold of $W_{0}^{s,p}(\Omega)$ of codimension 1, if i=1,2 and 2 if i=3 respectively.\\
(ii) The sets $K_{i}$ are complete, and for every $u\in M_{i}$ we have
\begin{align} 
T_{u}W_{0}^{s,p}(\Omega)=T_{u}M_{i}\oplus span\lbrace u_{+},u_{-}\rbrace
\end{align}
,where $T_{u}M$ is the tangent space at $u$ of the Banach manifold M.\\
(iii) The projection onto the first coordinate in (3.7) is uniformly continuous on bounded sets of $M_{i}$.
\end{lemma}

\begin{proof}
(i) Define $\tilde{M_{1}}=\lbrace u\in W_{0}^{s,p}(\Omega):\int_{\Omega}u_{+} > 0\rbrace, \tilde{M_{2}}=\lbrace u\in W_{0}^{s,p}(\Omega):\int_{\Omega}u_{-} > 0\rbrace$ and  $\tilde{M_{3}}=\tilde{M_{1}}\cap\tilde{M_{2}}$. It is easy to see that $\tilde{M_{i}}$ are open and $M_{i} \subset\tilde{M_{i}}$ and hence it is sufficient to show that $M_{i}$ is a regular sub-manifold of $W_{0}^{s,p}(\Omega)$. For this we will construct $C^{1}$ functions $\psi_{i}:\tilde{M_{i}}\rightarrow \mathbb{R}^{d}$, where $d$ denotes the codimension of $M_{i}$, such that $M_{i}$ is the inverse of a regular value of $\psi_{i}$. Define
\begin{align*}
\begin{split}
\psi_{1}(u)=&\int_{\mathbb{R}^N}\int_{\mathbb{R}^N}
\Phi(u_{+}(x)-u_{+}(y))(u_{+}(x)-u_{+}(y))K(x,y)dxdy-\int_\Omega |u_{+}|^{p^{\ast}_{s}}dx\\
&-\lambda \int_{\Omega}f(x,u)u_{+} dx.\\
\psi_{2}(u)=&\int_{\mathbb{R}^N}\int_{\mathbb{R}^N}
\Phi(u_{-}(x)-u_{-}(y))(u_{-}(x)-u_{-}(y))K(x,y)dxdy-\int_\Omega |u_{-}|^{p^{\ast}_{s}}dx\\
&-\lambda \int_{\Omega}f(x,u)u_{-} dx\\
\psi_{3}(u)=&(\psi_{1}(u),\psi_{2}(u)).
\end{split}
\end{align*}
It can be observed that $M_{i}=\psi_{i}^{-1}(\overrightarrow{0})$. Hence, we need to show that $\overrightarrow{0}$ is a regular value of $\psi_{i}$. We will use $(u+\epsilon u_{+})=u_{+}+\epsilon u_{+}$ for proving this. For $u \in M_{1}$,
\begin{align*}
\begin{split}
\langle \nabla \psi_{1}(u),u_{+}\rangle=&\frac{d}{d\epsilon}\psi_{1}(u+\epsilon u_{+})\\
=&\frac{d}{d\epsilon}\bigg[\int_{\mathbb{R}^N}\int_{\mathbb{R}^N}
\Phi((u+\epsilon u_{+})_{+}(x)-((u+\epsilon u_{+})_{+}(y))((u+\epsilon u_{+})_{+}(x)\\
&-(u+\epsilon u_{+})_{+}(y))K(x,y)dxdy-\int_{\Omega} |(u+\epsilon u_{+})_{+}|^{p^{\ast}_{s}}dx\\
&-\lambda \int_{\Omega}f(x,u+\epsilon u_{+})(u+\epsilon u_{+})_{+} dx\bigg] \\
=&\frac{d}{d\epsilon}\bigg[\int_{\mathbb{R}^N}\int_{\mathbb{R}^N}
\Phi((u_{+}+\epsilon u_{+})(x)-(u_{+}+\epsilon u_{+})(y))((u_{+}+\epsilon u_{+})(x)-(u_{+}+\epsilon u_{+})(y))\\
&K(x,y)dxdy-\int_{\Omega} |(u_{+}+\epsilon u_{+})|^{p^{\ast}_{s}}dx-\lambda \int_{\Omega}f(x,u+\epsilon u_{+})(u_{+}+\epsilon u_{+})dx\bigg] \\
=&\int_{\mathbb{R}^N}\int_{\mathbb{R}^N}\bigg[\Phi^{\prime}((u_{+}+\epsilon u_{+})(x)-((u_{+}+\epsilon u_{+})(y))(u_{+}(x)-u_{+}(y))((u_{+}+\epsilon u_{+})(x)\\
&-(u_{+}+\epsilon u_{+})(y))K(x,y)dxdy+\Phi((u_{+}+\epsilon u_{+})(x)-(u_{+}+\epsilon u_{+})(y))\\
&((u_{+}(x)-u_{+}(y))K(x,y)\bigg]-p^{\ast}_{s}(1+\epsilon)^{p^{\ast}_{s}}|u_{+}|^{p^{\ast}_{s}}dx-\lambda \int_{\Omega}f(x,u+\epsilon u_{+})u_{+}dx\\
&-\lambda\int_{\Omega}f_{u}(x,u+\epsilon u_{+})u_{+}^{2}dx
\end{split}
\end{align*}
Since $u\in M_{1}$ so
\begin{align*}
\begin{split}
\frac{d}{d\epsilon}\psi_{1}(u+\epsilon u_{+})\lvert_{\epsilon=0}& =\int_{\mathbb{R}^N}\int_{\mathbb{R}^N}\bigg[\Phi^{\prime}(u_{+}(x)-u_{+}(y))(u_{+}(x)-u_{+}(y))^{2}K(x,y)\\
&+\Phi( u_{+}(x)-u_{+}(y))(u_{+}(x)-u_{+}(y))K(x,y)\bigg]dxdy-\int_{\Omega}p^{\ast}_{s}|u_{+}|^{p^{\ast}_{s}}dx\\
&-\lambda\int_{\Omega}f(x,u)u_{+}dx-\lambda\int_{\Omega}f_{u}(x,u)u_{+}^{2}dx
\end{split}
\end{align*}
 Hence, (2.3) and (2.4)\\
\begin{align*}
\begin{split}
\frac{d}{d\epsilon}\psi_{1}(u+\epsilon u_{+})\lvert_{\epsilon=0}\;&\leq a\Lambda[u_{+}]^{p}_{s,p}+\Lambda^{2}[u_{+}]_{s,p}^{p}-\int_{\Omega}p^{\ast}_{s}|u_{+}|^{p^{\ast}_{s}}dx-\lambda\int_{\Omega}f(x,u)u_{+}dx\\
&-\lambda\int_{\Omega}f_{u}(x,u)u_{+}^{2}dx\\
&=(a\Lambda+\Lambda^{2})[u_{+}]_{s,p}^{p}-\int_{\Omega}p^{\ast}_{s}|u_{+}|^{p^{\ast}_{s}}dx-\lambda\int_{\Omega}f(x,u)u_{+}dx\\
&-\lambda\int_{\Omega}f_{u}(x,u)u_{+}^{2}dx
\end{split}
\end{align*}
We already have $\Lambda\in\left[1,\left(\frac{c_{2}}{p}\right)^{\frac{1}{4}}\right)$ from lemma 3.2. Choose ${\lq{a}\rq}$ such that 
$a\Lambda+\Lambda^{2} < p_{s}^{\ast}$. Hence,
\begin{align*}
\begin{split}
\frac{d}{d\epsilon}\psi_{1}(u+\epsilon u_{+})\lvert_{\epsilon=0}\;&\leq p^{\ast}_{s}\left([u_{+}]_{s,p}^{p}-\int_{\Omega}p^{\ast}_{s}|u_{+}|^{p^{\ast}_{s}}dx\right)-\lambda\int_{\Omega}f(x,u)u_{+}dx\\
&-\lambda\int_{\Omega}f_{u}(x,u)u_{+}^{2}dx\\
&= p^{\ast}_{s}\lambda\int_{\Omega}f(x,u)u_{+}dx-\lambda\int_{\Omega}f(x,u)u_{+}dx-\lambda\int_{\Omega}f_{u}(x,u)u_{+}^{2}dx\\
&= (p^{\ast}_{s}-1)\lambda\int_{\Omega}f(x,u)u_{+}dx-\lambda\int_{\Omega}f_{u}(x,u)u_{+}^{2}dx
\end{split}
\end{align*}
From the assumption $(f_{2})$, there exist $c_{1}\in\left(0,\frac{1}{p^{\ast}_{s}-1}\right)$ such that
\begin{align}
\lambda\int_{\Omega}f(x,u)u_{+}dx \leq c_{1}\lambda\int_{\Omega}f_{u}(x,u)u_{+}^{2}dx < \frac{1}{p^{\ast}_{s}-1}\lambda\int_{\Omega}f_{u}(x,u)u_{+}^{2}dx
\end{align}
Using (3.8) we conclude that 
$\frac{d}{d\epsilon}\psi_{1}(u+\epsilon u_{+})\lvert_{\epsilon=0} < 0$ which further implies $\langle\nabla\psi_{1}(u),u_{+}\rangle < 0$. Consequently, $\nabla\psi_{1}(u)\neq 0$.
Hence, $M_{1}$ is a regular sub-manifold of $W_{0}^{s,p}(\Omega)$.\\
In a similar manner it can be proved that $\langle\nabla\psi_{2}(u),u_{-}\rangle < 0$ for $u\in M_{2}$ implying that $M_{2}$ is also a regular sub-manifold of $W_{0}^{s,p}(\Omega)$.\\
Now if $\langle\nabla\psi_{1}(u),u_{-}\rangle = \langle\nabla\psi_{2}(u),u_{+}\rangle = 0$ for $u\in M_{3}$ then $\langle\nabla\psi_{1}(u),u\rangle < 0$ and $\langle\nabla\psi_{2}(u),u\rangle < 0$ implying that $\nabla\psi_{3}(u)\neq 0$ for $u\in M_{3}$. Therefore,
\begin{align*}
\begin{split}
\langle \nabla \psi_{1}(u),u_{-}\rangle=&\frac{d}{d\epsilon}\psi_{1}(u+\epsilon u_{+})\\
=&\frac{d}{d\epsilon}\bigg[\int_{\mathbb{R}^N}\int_{\mathbb{R}^N}
\Phi((u+\epsilon u_{-})_{+}(x)-(u+\epsilon u_{-})_{+}(y))((u+\epsilon u_{-})_{+}(x)\\
&-(u+\epsilon u_{-})_{+}(y))K(x,y)dxdy-\int_{\Omega} |(u+\epsilon u_{-})_{+}|^{p^{\ast}_{s}}dx\\
&-\lambda \int_{\Omega}f(x,u+\epsilon u_{-})(u+\epsilon u_{-})_{+} dx\bigg] \\
=& \frac{d}{d\epsilon}\bigg[\int_{\mathbb{R}^N}\int_{\mathbb{R}^N}
\Phi(u_{+}(x)-u_{+}(y))(u_{+}(x)-u_{+}(y))K(x,y)dxdy\\
&-\int_{\Omega} |u_{+}|^{p^{\ast}_{s}}dx-\lambda \int_{\Omega}f(x,u+\epsilon u_{-})u_{+}dx\bigg]\\
=& -\lambda\int_{\Omega}f_{u}(x,u+\epsilon u_{-})u_{+}u_{-}dx\\
=&\; 0.
\end{split}
\end{align*}
Similarly, $\langle\nabla\psi_{2}(u),u_{+}\rangle = 0$ and hence $M_{3}$ is also a regular sub-manifold of $W_{0}^{s,p}(\Omega)$.\\
(ii) To show that $K_{i}$ is complete. Let $(u_{n})$ be a Cauchy sequence in $K_{i}$ then $u_{n} \rightarrow u$ in $W_{0}^{s,p}(\Omega)$ which further implies that $(u_{n})_{\pm}\rightarrow u_{\pm}$. By using continuity of the functions, $(f_{2})$ and Lemma 3.3, we conclude that $u\in K_{i}$. Hence, $K_{i}$ is complete.\\
\hspace*{1cm}Now we show that $T_{u}W_{0}^{s,p}(\Omega)=T_{u}M_{1}\oplus span\lbrace u_{+}\rbrace$, where $M_{1}=\lbrace u:\psi_{1}(u)=0\rbrace$ and $T_{u}M_{1}=\lbrace v:\langle \nabla \psi_{1}(u),v\rangle=0\rbrace$. Let $v\in T_{u}W_{0}^{s,p}(\Omega)$ be a unit tangential vector, then $v=v_{1}+v_{2}$ where $v_{2}=\alpha u_{+}$ and $v_{1}=v-v_{2}$. Choose $\alpha=\frac{\langle \nabla \psi_{1}(u),u\rangle}{\langle \nabla \psi_{1}(u),u_{+}\rangle}$. From here we get $v_{1}\in T_{u}M_{1}$. Hence, $\langle \nabla \psi_{1}(u),v_{1}\rangle=0$. Similarly we can show that  $T_{u}W_{0}^{s,p}(\Omega)=T_{u}M_{2}\oplus span\lbrace u_{-}\rbrace$ and $T_{u}W_{0}^{s,p}(\Omega)=T_{u}M_{3}\oplus span\lbrace u_{+},u_{-}\rbrace$.\\
(iii) It is easy to see that the projection onto the first coordinate in (3.7) is uniformly continuous on bounded sets of $M_{i}$.
\end{proof}
\noindent We next prove the following lemma.
\begin{lemma}\label{lem5}
The unrestricted functional $I$ verifies the Palais-Smale condition for energy level $c < \left(\frac{1}{c_{2}}-\frac{1}{p_{s}^{\ast}}\right)\frac{S^{\frac{N}{ps}}}{\Lambda^{2}}$, where $S=\underset{\phi\in C_{c}^{\infty}(\Omega)}{inf}\frac{{\|\phi\|}_{s,p}^{p}}{{\|\phi\|}_{p^{\ast}_{s}}^{p}}$ is the best Sobolev constant for the fractional Laplacian.
\end{lemma}

\begin{proof}
Let $(u_{n})$ be a Palais-Smale sequence. Then $I(u_{n})\rightarrow c$ and $I^{\prime}(u_{n})\rightarrow 0$. Let us assume that $\|u_{n}\|\rightarrow \infty$ as $n \rightarrow\infty$. We have
\begin{align*}
\begin{split}
c&\geq \underset{n\rightarrow\infty}{\text {lim}} I(u_{n})-\frac{1}{c_{2}}\langle I^{\prime}(u_{n}),u_{n}\rangle\\
&=\underset{n\rightarrow\infty}{\text {lim}}\int_{\mathbb{R}^N}\int_{\mathbb{R}^N}
\mathcal{P}\Phi(u_{n}(x)-u_{n}(y))K(x,y)dx dy -\frac{1}{p_{s}^{\ast}}\int_\Omega |u_{n}|^{p_{s}^{\ast}}-\lambda \int_{\Omega} F(x,u_{n})dx\\
&-\frac{1}{c_{2}}\left(\int_{\mathbb{R}^N}\int_{\mathbb{R}^N}
\Phi(u_{n}(x)-u_{n}(y))(u_{n}(x)-u_{n}(y))K(x,y)dxdy -\int_\Omega |u_{n}|^{p^{\ast}_{s}}dx\right.\\
&\left.-\lambda \int_{\Omega}f(x,u_{n})u_{n}dx\right)\\
&\geq \underset{n\rightarrow\infty}{\text {lim}}\;\frac{1}{p\Lambda^{2}}\int_{\mathbb{R}^N}\int_{\mathbb{R}^N}\frac{|u_{n}(x)-u_{n}(y)|^{p}}{|x-y|^{N+ps}}dx dy-\frac{1}{p_{s}^{\ast}}\int_{\Omega}|u_{n}|^{p_{s}^{\ast}}-\lambda\int_{\Omega}F(x,u_{n})dx+\frac{1}{c_{2}}\int_{\Omega}|u_{n}|^{p_{s}^{\ast}}\\
&-\frac{\Lambda^{2}}{c_{2}}\int_{\mathbb{R}^N}\int_{\mathbb{R}^N}\frac{|u_{n}(x)-u_{n}(y)|^{p}}{|x-y|^{N+ps}}dx dy+\frac{\lambda}{c_{2}}\int_{\Omega}f(x,u_{n})u_{n}dx\\
&=\underset{n\rightarrow\infty}{\text {lim}}\left(\frac{1}{\Lambda^{2}p}-\frac{\Lambda^2}{c_{2}}\right){\|u_{n}\|}_{s,p}^{p}+\left(\frac{1}{c_{2}}-\frac{1}{p^{\ast}_{s}}\right)\int_{\Omega}|u_{n}|^{p_{s}^{\ast}}dx+\frac{\lambda}{c_{2}}\int_{\Omega}[f(x,u_{n})u_{n}-c_{2}F(x,u_{n})]dx
\end{split}
\end{align*}
Using $(f_{2})$ and the fact that $c_{2} < p_{s}^{\ast}$ we get 
$c > \left(\frac{1}{\Lambda^{2}p}-\frac{\Lambda^{2}}{c_{2}}\right){\|u_{n}\|}_{s,p}^{p}$. Thus, we get a contradiction for $\left(\frac{1}{\Lambda^{2}p}-\frac{\Lambda^{2}}{c_{2}}\right) > 0$.
Hence, $ (u_{n})$ is bounded in $W_{0}^{s,p}(\Omega)$ whenever $\Lambda\in\left[1,\left(\frac{c_{2}}{p}\right)^{\frac{1}{4}}\right)$. By the reflexivity property of $W_{0}^{s,p}(\Omega)$ there exists a weakly convergent subsequence say $(u_{n})\rightharpoonup u$ in $W_{0}^{s,p}(\Omega)$. This implies
$\Rightarrow u_{n}\rightarrow u$ in $L^{q}(\Omega)\;; q\in[1,p_{s}^{\ast})$ as $n \rightarrow\infty$.\\
In order to prove the strong convergence in $L^{p_{s}^{\ast}}(\Omega)$, we will use concentration compactness principle for the fractional sobolev space which is as follows [refer \cite{Mosconi}].
\begin{theorem}
Let $(u_{n})$ be a bounded sequence in $W_{0}^{s,p}(\Omega)$. Then upto a subsequence there exists $u\in W_{0}^{s,p}(\Omega)$, Borel measures $\mu$ and $\nu$, $J$ denumerable, $x_{j}\in \overline{\Omega}, \nu_{j} > 0, \mu_{j}\geq 0$ with $\mu_{j}+\nu_{j} > 0$ for $j\in J$ and $u_{n}\rightharpoonup u$ in $W_{0}^{s,p}(\Omega)$ and is strongly convergent in $L^{p}(\Omega)$.
$|D^{s}u_{n}|^{p}\overset{\ast}{\rightharpoonup}\mu, |u_{n}|^{p_{s}^{\ast}}\overset{\ast}{\rightharpoonup}\nu, d\mu \geq |D^{s}u|^{p}+\sum_{j\in J}\mu_{j}\delta_{x_{j}}, d\nu=|u_{n}|^{p_{s}^{\ast}}+\sum_{j\in J}\nu_{j}\delta_{x_{j}}$ and $\mu_{j}\geq S{\nu_{j}}^{\frac{p}{p_{s}^{\ast}}}$. Here, $\mu_{j}=\mu({x_{j}})$ and $\nu_{j}=\mu({x_{j}})$.
\end{theorem}
Hence, if we show that the index set $J$ is empty then $u_{n}\rightarrow u$ in $L^{p^{\ast}_{s}}(\Omega)$. Let $\xi$ be a unit ball such that $\xi_{j,\delta}(x)=\xi\left(\frac{|x-x_{j}|}{\delta}\right)$ where $\xi$ has a compact support in $B_{2}, \xi=1$ in $B_{1}$ and $0\leq\xi\leq 1$. Since $(u_{n})$ is a PS sequence corresponding to the functional $I$, hence $I^{\prime}(u_{n})\rightarrow 0$. Therefore, on considering test function as $\xi_{j,\delta}u_{n}$ we have $\langle I^{\prime}(u_{n}),\xi_{j,\delta}u_{n}\rangle \rightarrow 0$. Now,
\begin{align}
\begin{split}
\langle I^{\prime}(u_{n}),\xi_{j,\delta}u_{n}\rangle=&\int_{\mathbb{R}^{N}}\int_{\mathbb{R}^{N}}\Phi(u_{n}(x)-u_{n}(y))(\xi_{j,\delta}u_{n}(x)-\xi_{j,\delta}u_{n}(y))K(x,y)dx dy\\
&-\int_{\Omega}|u_{n}|^{p_{s}^{\ast}-2}u_{n}\xi_{j,\delta}dx-\lambda\int_{\Omega}f(x,u_{n})\xi_{j,\delta}u_{n} dx
\end{split}
\end{align}
The first term of the RHS of equation (3.9) can be written as 
\begin{align*}
\begin{split}
&\int_{\mathbb{R}^{N}}\int_{\mathbb{R}^{N}}\Phi(u_{n}(x)-u_{n}(y))(\xi_{j,\delta}u_{n}(x)-\xi_{j,\delta}u_{n}(y))K(x,y)dx dy=\\
&\int_{\mathbb{R}^{N}}\int_{\mathbb{R}^{N}}\Phi(u_{n}(x)-u_{n}(y))(u_{n}(x)-u_{n}(y))\xi_{j,\delta}(x)K(x,y)dx dy\\
+&\int_{\mathbb{R}^{N}}\int_{\mathbb{R}^{N}}\Phi(u_{n}(x)-u_{n}(y))(\xi_{j,\delta}(x)-\xi_{j,\delta}(y))u_{n}(y)K(x,y)dx dy\\
=&I+II
\end{split}
\end{align*}
We further have,
\begin{align*}
\begin{split}
I&=\int_{\mathbb{R}^{N}}\int_{\mathbb{R}^{N}}\Phi(u_{n}(x)-u_{n}(y))(u_{n}(x)-u_{n}(y))\xi_{j,\delta}(x)K(x,y)dx dy\\
&\geq \Lambda^{-2}\int_{\mathbb{R}^N}\int_{\mathbb{R}^N}\frac{|u_{n}(x)-u_{n}(y)|^{p}}{|x-y|^{N+ps}}\xi_{j,\delta}(x)dx dy\\
&=\Lambda^{-2}\int_{\mathbb{R}^{N}}|D^{s}u_{n}|^{p}\xi_{j,\delta}(x)dx
\end{split}
\end{align*}
So, $\underset{\delta\rightarrow 0}{\text{lim}}\;\underset{n\rightarrow\infty}{\text{lim}}I \geq \Lambda^{-2}\mu_{j}$. From \cite{Bonder}, it can be seen that $\underset{\delta\rightarrow 0}{\text{lim}}\;\underset{n\rightarrow\infty}{\text{lim}}II=0$. Hence, for the first term of RHS in (3.9) we have
\begin{align*}
\underset{\delta\rightarrow 0}{\text{lim}}\;\underset{n\rightarrow\infty}{\text{lim}}\int_{\mathbb{R}^{N}}\int_{\mathbb{R}^{N}}\Phi(u_{n}(x)-u_{n}(y))(\xi_{j,\delta}u_{n}(x)-\xi_{j,\delta}u_{n}(y))K(x,y)dx dy \geq \Lambda^{-2}\mu_{j}
\end{align*}
Also, $\underset{\delta\rightarrow 0}{\text{lim}}\;\underset{n\rightarrow\infty}{\text{lim}}\int_{\Omega}|u_{n}|^{p_{s}^{\ast}}\xi_{j,\delta}dx=\nu_{j}$ and $\underset{\delta\rightarrow 0}{\text{lim}}\;\underset{n\rightarrow\infty}{\text{lim}}\int_{\Omega}f(x,u_{n})u_{n}\xi_{j,\delta}dx=0$.\\
Thus, from (3.9), we have $\Lambda^{-2}\mu_{j}-\nu_{j}\leq 0$ which further implies $\mu_{j}\leq \frac{\nu_{j}}{\Lambda^2}$. On using the relation $\mu_{j}\geq S{\nu_{j}^{\frac{p}{p^{\ast}_{s}}}}$ we get either $\nu_{j}=0$ or $\nu_{j}\geq \frac{S^{\frac{N}{ps}}}{\Lambda^2}$. Let us assume that $\nu_{j}\geq \frac{S^{\frac{N}{ps}}}{\Lambda^2}$, then using the fact that norm is weakly lower semi-continuous we have
\begin{align*}
\begin{split}
c=&\underset{n\rightarrow\infty}{\text{lim}}I(u_{n})\\
=& \underset{n\rightarrow\infty}{\text{lim}}I(u_{n})-\frac{1}{c_{2}}\langle I^{\prime}(u_{n}),u_{n}\rangle\\
\geq& \underset{n\rightarrow\infty}{\text{lim}}\left(\frac{1}{c_{2}}-\frac{1}{p^{\ast}_{s}}\right)\int_{\Omega}|u_{n}|^{p_{s}^{\ast}}dx\\
\geq& \underset{n\rightarrow\infty}{\text{lim inf}}\left(\frac{1}{c_{2}}-\frac{1}{p^{\ast}_{s}}\right)\int_{\Omega}|u_{n}|^{p_{s}^{\ast}}dx\\
\geq&\left(\frac{1}{c_{2}}-\frac{1}{p^{\ast}_{s}}\right)\left(\int_{A_{\epsilon}}|u|^{p_{s}^{\ast}}+\sum_{j\in J}\nu_{j}\right)\\
\geq& \left(\frac{1}{c_{2}}-\frac{1}{p^{\ast}_{s}}\right)\nu_{j}\\
\geq&\left(\frac{1}{c_{2}}-\frac{1}{p^{\ast}_{s}}\right)\frac{S^{\frac{N}{ps}}}{\Lambda^2}
\end{split}
\end{align*}
Hence, for $c < \left(\frac{1}{c_{2}}-\frac{1}{p^{\ast}_{s}}\right)\frac{S^{\frac{N}{ps}}}{\Lambda^2}$, the index set $J$ is empty as $\nu_{j}=0$ and since $\mu_{j}\leq \frac{\nu_{j}}{\Lambda^2}$ so $\mu_{j}=0$. Therefore $\int_{\Omega}|u_{n}|^{p^{\ast}_{s}}\rightarrow \int_{\Omega}|u|^{p^{\ast}_{s}}$. We also know that $W_0^{s, p}(\Omega)$ is compactly embedded in
$L^r(\Omega)$, $r\in[1,p_s^*)$ and hence $u_n\rightarrow u$ in
$L^r(\Omega)$, for $1\leq r <p_s^*$. Consider $\tilde{u}_n = u_n-u$.
Then $ \tilde{u}_n \rightharpoonup 0$ in $W_0^{s, p}(\Omega)$.
Consider,
\begin{align}
\begin{split}
\langle I'(\tilde{u}_n),\tilde{u}_n\rangle
&=\int_{\mathbb{R}^N}\int_{\mathbb{R}^N}\Phi(\tilde{u}_n(x)-\tilde{u}_n(y))(\tilde{u}_n(x)-\tilde{u}_n(y))K(x,y)dxdy\\
&-\lambda\int_\Omega |\tilde{u}_n|^{p^{\ast}_{s}}dx-\int_\Omega
f(x,\tilde{u}_n)\tilde{u}_n dx
\end{split}
\end{align}
The second and the third term of the functional approach to $0$ as
$n\rightarrow\infty$ since $u_n\rightarrow u$ in $L^{p_s^*}(\Omega)$.  Also by $(f_{2})$, we see that the third term in (3.10) will also converge to $0$. Hence from (3.10), we have first term in RHS of (3.10) will converge to zero.\\
But
\begin{align}
0 \leq \frac{1}{\Lambda^{2}}\|\tilde{u}_n\|^p_{W_0^{s,
p}(\Omega)}\leq
\int_{\mathbb{R}^N}\int_{\mathbb{R}^N}\Phi(\tilde{u}_n(x)-\tilde{u}_n(y))(\tilde{u}_n(x)-\tilde{u}_n(y))K(x,y)dxdy
\end{align}
From (3.11) it is clear that $\displaystyle{\lim_{n\rightarrow
\infty}\|\tilde{u}_n\|_{W_0^{s, p}(\Omega)}}=0$, that is
$\tilde{u}_n \rightarrow 0$ in $W_0^{s, p}(\Omega)$. Hence
$u_n\rightarrow u$ in $W_0^{s, p}(\Omega)$ and the unrestricted functional $I$ satisfies Palais-Smale condition.
\end{proof}
\noindent We now give a lemma which is a consequence of the Lemma 3.6.
\begin{lemma}
The functional $I|_{K_{i}}$ satisfies the Palais-Smale condition for every energy level $c < \left(\frac{1}{c_{2}}-\frac{1}{p^{\ast}_{s}}\right)\frac{S^{\frac{N}{ps}}}{\Lambda^2}$.  Then $u$ is also a critical point of the functional $I$ and hence a solution of the problem (1.1).
\end{lemma}

\noindent {\it{Proof of theorem 1.1.}} Since $I$ is bounded below on $K_{1}$. Therefore by Ekeland's variational principle, there exists $v_{k}\in K_{1}$ such that $I(v_{k})\rightarrow c:= \underset{K_{1}}{\text{inf}}$ and $(I|_{K_{1}})^{\prime}(v_{k})\rightarrow 0$.\\
We have to check if we choose $\lambda$ large, we have $c < \left(\frac{1}{c_{2}}-\frac{1}{p^{\ast}_{s}}\right)\frac{S^{\frac{N}{ps}}}{\Lambda^2}$. It follows from Lemma 3.1. For instance, for $c$ we have that for choosing $w_{0} \geq 0, c\leq I(t_{\lambda}w_{0})\leq \frac{t_{\lambda}^{p}\Lambda^{2}}{p}[w_{0}]_{s,p}^{p}$. Also, it follows from the estimate of $t_{\lambda}$ in Lemma 3.1 that $c\rightarrow 0$ as $\lambda\rightarrow\infty$.Hence, there exists $\lambda^{\prime}(p.p_{s}^{\ast},n,c_{3})$ such that for all $\lambda > \lambda^{\prime}, c< \left(\frac{1}{c_{2}}-\frac{1}{p^{\ast}_{s}}\right)\frac{S^{\frac{N}{ps}}}{\Lambda^2}$. The other cases are analogous. From the construction of $K_{i}$ it can be seen that one solution is negative, one is positive and one is sign changing. Hence the theorem is proved.

\section*{Conclusion}
The problem involving an integro-differential operator and a critical exponent has been studied. As a result of this study we conclude the existence of three nontrivial solutions.

\section*{Acknowledgement}
The author  Amita Soni thanks the Department of
Science and Technology (D. S. T), Govt. of India for financial
support. Both the authors also acknowledge the facilities received
from the Department of mathematics, National Institute of Technology
Rourkela.

\bibliographystyle{apa}

{\sc Amita Soni} and {\sc D. Choudhuri}\\
Department of Mathematics,\\
National Institute of Technology Rourkela, Rourkela - 769008,
India\\
e-mails: soniamita72@gmail.com and dc.iit12@gmail.com.
\end{document}